\documentclass[a4paper,leqno]{amsart}

\usepackage[english]{babel} 									
\usepackage[T1]{fontenc}							
\usepackage[utf8]{inputenx}
				
\usepackage{mathtools}								
\usepackage{empheq}									
\usepackage{enumitem}								

\usepackage{amssymb}								
\usepackage[sc]{mathpazo}							
\usepackage{mathrsfs}								
\usepackage{bm}                                     
				
\usepackage{verbatim}                               
\usepackage{iflang}									
\usepackage{xifthen}								

\usepackage[final]{pdfpages}						
\usepackage{todonotes}								
\usepackage{aliascnt}								
\usepackage{needspace}								
\usepackage{refcount}

\usepackage{subcaption}			
\usepackage{placeins} 			
\usepackage{grffile}			
\usepackage[all]{xy}			

\usepackage{color}				

\usepackage[babel]{csquotes}							

\usepackage{nameref}									
\usepackage[colorlinks	=	true,						
          	 	linkcolor	=	red,
            	urlcolor	=	red,
            	citecolor	=	red]%
            	{hyperref}
\usepackage{bookmark}									
\usepackage{cleveref}									
\numberwithin{equation}{section}

\newtheorem{theorem}{Theorem}[section]

\newtheorem{corollary}[theorem]{Corollary}
\newtheorem{lemma}[theorem]{Lemma}
\newtheorem{proposition}[theorem]{Proposition}
\newtheorem{conjecture}[theorem]{Conjecture}

\newtheorem{definition}[theorem]{Definition}

\theoremstyle{remark}
\newtheorem{remark}[theorem]{Remark}

\theoremstyle{remark}

\DeclareMathOperator{\diam}{diam}

\newcommand{\apmd}[2][]{							
	\ifthenelse{\equal{#1}{}}%
					{ \operatorname{N}_{#2}	}%
					{ \operatorname{N}_{#1,#2} 	}}

\newcommand{\aint}[2][]{
	\ifthenelse{\equal{#1}{}}%
					{%
\mathchoice%
      {\mathop{\kern 0.2em\vrule width 0.6em height 0.69678ex depth -0.58065ex
              \kern -0.8em \intop}\nolimits_{\kern -0.45em#2}^{#1}}%
      {\mathop{\kern 0.1em\vrule width 0.5em height 0.69678ex depth -0.60387ex
              \kern -0.6em \intop}\nolimits_{#2}^{#1}}%
      {\mathop{\kern 0.1em\vrule width 0.5em height 0.69678ex depth -0.60387ex
              \kern -0.6em \intop}\nolimits_{#2}^{#1}}%
      {\mathop{\kern 0.1em\vrule width 0.5em height 0.69678ex depth -0.60387ex
              \kern -0.6em \intop}\nolimits_{#2}^{#1}}}%
					{%
\mathchoice%
      {\mathop{\kern 0.2em\vrule width 0.6em height 0.69678ex depth -0.58065ex                                              
              \kern -0.8em \intop}\nolimits_{\kern -0.45em#1}^{#2}}%
      {\mathop{\kern 0.1em\vrule width 0.5em height 0.69678ex depth -0.60387ex
              \kern -0.6em \intop}\nolimits_{#1}^{#2}}%
      {\mathop{\kern 0.1em\vrule width 0.5em height 0.69678ex depth -0.60387ex
              \kern -0.6em \intop}\nolimits_{#1}^{#2}}%
      {\mathop{\kern 0.1em\vrule width 0.5em height 0.69678ex depth -0.60387ex
              \kern -0.6em \intop}\nolimits_{#1}^{#2}}}}

\begin{document}
\title[Removability of singularies]{Quasiregular curves: Removability of singularities}

\author{Toni Ikonen} 
\address{Department of Mathematics and Statistics, P. O. Box 68 (Pietari Kalmin katu 5), FI-00014 University of Helsinki, Finland}

\email{toni.ikonen@helsinki.fi}

\thanks{This work was supported by the Academy of Finland, project number 332671. \newline {\it 2020 Mathematics Subject Classification.} Primary: 30C65. Secondary: 49Q15, 53C65}


\begin{abstract}
We prove a Painlevé theorem for bounded quasiregular curves in Euclidean spaces extending removability results for quasiregular mappings due to Iwaniec and Martin. The theorem is proved by extending a fundamental inequality for volume forms to calibrations and proving a Caccioppoli inequality for quasiregular curves.

We also establish a qualitatively sharp removability theorem for quasiregular curves whose target is a Riemannian manifold with sectional curvature bounded from above and injectivity radius bounded from below. As an application, we extend a theorem of Bonk and Heinonen for quasiregular mappings to the setting of quasiregular curves: every non-constant quasiregular $\omega$-curve from $\mathbb{R}^n$ into $( N, \omega )$, where the bounded cohomology class of $\omega$ is in the bounded Künneth ideal, has infinite energy.
\end{abstract}

\maketitle\thispagestyle{empty}

\section{Introduction}\label{sec:intro}

A continuous mapping $f \colon M \rightarrow N$ between oriented Riemannian $n$-manifolds is a \emph{$K$-quasiregular mapping} if
\begin{equation}\label{eq:quasiregular}
    f \in W^{1,n}_{loc}( M, N )
    \quad\text{and}\quad
    \| Df \|^n \leq K \star f^{\star}\mathrm{vol}_{N}
    \quad\text{almost everywhere};
\end{equation}
here $\|Df\|$ is the operator norm of the weak differential of $f$, $f^{\star}\mathrm{vol}_N$ is the pullback of the Riemannian volume form of $N$, and $\star f^{\star}\mathrm{vol}_{N}$ its Hodge star dual. We recall that $\star f^{\star}\mathrm{vol}_N$ is the Jacobian of $f$.

Similarly, a continuous mapping $f \colon M \rightarrow N$ from an oriented surface into a Kähler manifold, with symplectic form $\omega_{sym}$, is a \emph{holomorphic curve} if
\begin{equation}\label{eq:holomorphic}
    f \in W^{1,2}_{loc}( M, N )
    \quad\text{and}\quad
    \| Df \|^2 = \star f^{\star}\omega_{sym}
    \quad\text{almost everywhere}.
\end{equation}
Recall that $d\omega_{sym} = 0$ and $\omega_{sym}(x) \neq 0$ for every $x \in N$.

Motivated by the similarity between \eqref{eq:quasiregular} and \eqref{eq:holomorphic}, Pankka coined the term \emph{quasiregular curve} in \cite{Pan:20}: Consider an oriented Riemannian $m$-manifold and $\omega \in \Omega^{n}( N )$ for $2 \leq n \leq m$  such that $\omega$ is closed. Then a continuous $f \colon M \rightarrow N$ from an oriented Riemannian $n$-manifold $M$ is a \emph{$K$-quasiregular $\omega$-curve} if
\begin{equation}\label{eq:quasiregular:curve}
    f \in W^{1,n}_{loc}( M, N )
    \quad\text{and}\quad
    ( \| \omega \| \circ f ) \| Df \|^n \leq K \star f^{\star}\omega
    \quad\text{almost everywhere}.
\end{equation}
Here
\begin{equation*}
    \| \omega \|( x )
    =
    \sup
    \left\{
        \omega( v_1 \wedge \dots \wedge v_n )
        \colon
        ( v_1, \dots, v_n )
        \text{ are orthonormal in $T_x N$}
    \right\}
\end{equation*}
is the \emph{comass} of $\omega$. In \eqref{eq:quasiregular} and \eqref{eq:holomorphic}, the comasses of the volume form and the symplectic forms are one. Further examples of interest for $\omega$ include the calibrations in the sense of \cite{Har:Law:82} and closed forms $\omega \in \Omega^{n}(N)$ when $N$ is compact. We mainly consider \emph{bounded} (i.e. $\sup \|\omega\| < \infty$) and \emph{strongly nonvanishing} (i.e. $\inf \|\omega\| > 0$) closed forms.

\subsection{Painlevé problem for quasiregular curves}\label{sec:painleve}
A closed set $E \subset \mathbb{C}$ is \emph{removable for bounded holomorphic maps} if for every open set $U \subset \mathbb{C}$ and every bounded holomorphic map $\phi \colon U \setminus E \rightarrow \mathbb{C}$, there is a holomorphic map $\Phi \colon U \to \mathbb{C}$ for which $\Phi|_{ U \setminus E } = \phi$. The Painlevé problem asks for necessary and sufficient geometric conditions on $E$ that guarantee the removability of $E$ for bounded holomorphic maps. The problem has been resolved by Tolsa \cite{Tol:03} in terms of curvatures of measures. Similar problems have been considered for quasiregular mappings. In fact, Iwaniec and Martin conjectured the following for quasiregular mappings.
\begin{conjecture}[\cite{Iwa:Mar:93}]\label{conj:remov:bounded}
Sets of Hausdorff $d$-measure zero, $d = n/(1+K) \leq n/2$, are removable under bounded $K$-quasiregular mappings. 
\end{conjecture}
A closed set $E \subset \mathbb{R}^n$ is \emph{removable for bounded $K$-quasiregular mappings} if for every open $U \subset \mathbb{C}$ and every bounded $K$-quasiregular $f \colon U \setminus E \rightarrow \mathbb{C}$, there is a $K$-quasiregular $\Phi \colon U \rightarrow \mathbb{C}$ for which $\Phi|_{ U \setminus E } = \phi$.

\Cref{conj:remov:bounded} is well-understood when $n = 2$. Indeed, bounded $1$-quasiregular mappings coincide with holomorphic mappings by Weyl's lemma, cf. \cite[p. 27]{Ast:Iwa:Mar:09}, so \Cref{conj:remov:bounded} is related to the Painlevé problem above. In regards to the Painlevé problem, sets of zero length are known to be removable for holomorphic mappings while sets whose Hausdorff dimension is strictly bigger than one are never removable (see \cite[p.371]{Dav:98}). When $K > 1$, it is known that sets with $\sigma$-finite $2/(1+K)$-dimensional Hausdorff measure are removable for bounded $K$-quasiregular mappings, cf. \cite[Theorem 1.2]{Ast:Clo:Oro:UT:08}.  We refer to the introduction of \cite{Ast:Clo:Oro:UT:08} for further background.

We now move to higher dimensions and recall the following theorem due to Iwaniec and Martin.
\begin{theorem}[Theorem 17.3.1, \cite{Iw:Ma:01}]\label{thm:remov:bounded}
There is $\lambda(n) \geq 1$ such that if $E \subset \mathbb{R}^n$ is closed and the Hausdorff dimension satisfies
\begin{equation}\label{eq:hausdorffbound:prev}
    \mathrm{dim}_{ \mathcal{H} }( E )
    <
    \frac{ n }{ 1 + K \lambda(n) },
\end{equation}
then $E$ is removable for bounded $K$-quasiregular mappings.
\end{theorem}
The number $\lambda(n)$ appears in \cite[Theorem 7.8.2]{Iw:Ma:01} and Iwaniec and Martin have conjectured that $\lambda(n) = 1$, see \cite[p. 164]{Iw:Ma:01}. If the equality $\lambda(n) = 1$ would hold, the upper bound in \Cref{thm:remov:bounded} would coincide with the critical Hausdorff dimension in \Cref{conj:remov:bounded}. The same constant appears in the following extension of \Cref{thm:remov:bounded}.
\begin{theorem}\label{thm:remov:bounded:curve}
Let $\omega \in \Lambda^{n} \mathbb{R}^m$ be a calibration, and let $\lambda(n) \geq 1$ be the constant from \Cref{thm:remov:bounded}. If $E \subset \mathbb{R}^n$ is closed and the Hausdorff dimension satisfies
\begin{equation}\label{eq:hausdorffbound}
    \mathrm{dim}_{ \mathcal{H} }( E )
    <
    \frac{ n }{ 1 + K \lambda(n) \binom{m}{n} },
\end{equation}
then every bounded $K$-quasiregular $\omega$-curve $f \colon U \setminus E \rightarrow \mathbb{R}^m$, for open $U \subset \mathbb{R}^n$, is the restriction of a $K$-quasiregular $\omega$-curve $F \colon U \rightarrow \mathbb{R}^m$.
\end{theorem}
The notation $\Lambda^{n} \mathbb{R}^m$ refers to $n$-covectors and the $n$-covectors are identified with \emph{constant-coefficient} $n$-forms.

The proof of \Cref{thm:remov:bounded:curve} has a similar strategy as the proof of \Cref{thm:remov:bounded}. To set the stage, we define the $\ell^1$-norm $\| \omega \|_{ \ell^1 } \coloneqq \sum_{ I } |\omega_I|$ of a calibration $\omega = \sum_{ I } \omega_I dx^{I}$ expressed in the standard basis of $n$-covectors $\Lambda^{n} \mathbb{R}^m$; the $\ell^1$-norm of a calibration satisfies $1 \leq \|\omega\|_{ \ell^1 } \leq \binom{m}{n}$ and the upper bound is related to the dimension bound \eqref{eq:hausdorffbound}. We prove a Caccioppoli inequality for quasiregular curves in the interval defined by the critical exponents
\begin{align}\label{eq:sharpexponents:intro}
    q(K,n,\|\omega\|_{ \ell^1 } )
    &=
    \frac{ K \lambda(n) \|\omega\|_{ \ell^1 }  }{ K \lambda(n) \|\omega\|_{ \ell^1 } + 1 }n
    \quad\text{and}\quad
    \\ \notag
    p(K,n,\|\omega\|_{ \ell^1 } ) &= \frac{ K \lambda(n) \|\omega\|_{ \ell^1 } }{ K \lambda(n) \|\omega\|_{ \ell^1 } - 1 }n.
\end{align}
More precisely, we have the following theorem.
\begin{theorem}[Caccioppoli inequality]\label{thm:caccioppoli:estimate}
Consider $K \geq 1$, a calibration $\omega \in \Lambda^{n} \mathbb{R}^m$, and $s \in ( q(K,n,\|\omega\|_{ \ell^1 } ), p(K,n,\|\omega\|_{ \ell^1 } ) )$. Suppose that $f \in W^{1,s}_{loc}( \Omega, \mathbb{R}^m )$ satisfies
\begin{equation*}
    \|Df\|^n
    \leq
    K \star f^{\star}\omega
    \quad\text{almost everywhere.}
\end{equation*}
Then, for every nonnegative $\phi \in \mathcal{C}^{\infty}_c( \Omega )$,
\begin{align*}
    \| \phi Df \|_{ L^{s}( \Omega, \mathbb{R}^{ m \times n } ) }
    &\leq
    C
    \| f \otimes D\phi \|_{ L^{s}( \Omega, \mathbb{R}^{ m \times n } ) }
\end{align*}
for a constant $C$ depending only on $s, n, \|\omega\|_{ \ell^1 }$ and $K$.
\end{theorem}
The notation $f \otimes D\phi$ refers to the \emph{tensor product} of the column vector $f$ and the row vector $D\phi$ and is identified with a section of matrices $\mathbb{R}^{ m \times n }$.

The interval $( q(K,n,1 ), p(K,n,1 ) )$ appears in the Caccioppoli inequality for quasiregular mappings \cite[Theorem 14.4.1]{Iw:Ma:01}. While the statement and the proof of \Cref{thm:caccioppoli:estimate} is similar to that of \cite[Theorem 14.4.1]{Iw:Ma:01}, key steps in the proof require careful modifications when $m > n$ even if $\|\omega\|_{ \ell^1 } = 1$.

\Cref{thm:caccioppoli:estimate} was previously known for quasiregular curves for the natural exponent $s = n$ \cite{Onn:Pan:21}. Moreover, \Cref{thm:caccioppoli:estimate} extends earlier results by Iwaniec and Martin \cite{Iwa:Mar:93,Iw:Ma:01} and Faraco and Zhong \cite{Far:Zho:06} for mappings.

Onninen and Pankka used the Caccioppoli inequality to prove a weak reverse Hölder inequality for quasiregular curves. In a similar way, we use \Cref{thm:caccioppoli:estimate} to derive a weak reverse Hölder inequality for quasiregular curves for exponents in the critical interval $( q(K,n,\|\omega\|_{ \ell^1 } ), p(K,n,\|\omega\|_{ \ell^1 } ) )$. We recall that a weak reverse Hölder inequality leads to higher Sobolev regularity for $f$ via Generalized Gehring's lemma. This leads to the following theorem.
\begin{theorem}[Higher Sobolev regularity]\label{thm:integrationbyparts}
Consider $K \geq 1$, a calibration $\omega \in \Lambda^{n} \mathbb{R}^m$, and $s \in ( q(K,n,\|\omega\|_{ \ell^1 } ), p(K,n,\|\omega\|_{ \ell^1 } ) )$. Suppose that $f \in W^{1,s}_{loc}( \Omega, \mathbb{R}^m )$ satisfies
\begin{equation*}
    \|Df\|^n
    \leq
    K \star f^{\star}\omega
    \quad\text{almost everywhere.}
\end{equation*}
Then $f \in W^{1,s'}_{loc}( \Omega, \mathbb{R}^m )$ for every $s' \in ( q(K,n,\|\omega\|_{ \ell^1 }), p(K,n,\|\omega\|_{ \ell^1 }) )$.
\end{theorem}
\Cref{thm:remov:bounded:curve} follows from the inequality $\|\omega\|_{ \ell^1 } \leq \binom{m}{n}$ and Theorems \ref{thm:caccioppoli:estimate} and \ref{thm:integrationbyparts} with standard arguments.

Onninen and Pankka proved a Hölder continuity estimate with the exponent $1/( K \|\omega\|_{ \ell^1 } )$ for calibrations $\omega \in \Lambda^{n} \mathbb{R}^m$ and quasiregular $\omega$-curves \cite[Theorem 1.1]{Onn:Pan:21}. Later, the sharp exponent $1/K$ was proved by the author, cf. \cite[Theorem 1.5]{Iko:23}, using pushforward of integral currents under Sobolev maps and Almgren's isoperimetric inequality for integral currents \cite{Alm:86}; in particular the dependence on $\|\omega\|_{ \ell^1 }$ was removed. Using the techniques from \cite{Iko:23}, we prove a similar improvement on the upper bound in \Cref{thm:integrationbyparts}.
\begin{theorem}\label{thm:pushforward}
Consider $K \geq 1$ and $n \geq 2$. Then there is $r = r(K,n) > n$ such that if $\omega \in \Lambda^n \mathbb{R}^m$ is a calibration and $f \in W^{1,n}_{loc}( \Omega, \mathbb{R}^m )$ satisfies
\begin{equation*}
    \|Df\|^n
    \leq
    K \star f^{\star}\omega
    \quad\text{almost everywhere,}
\end{equation*}
then $f \in W^{1,s'}_{loc}( \Omega, \mathbb{R}^m )$ for every $s' \in [n, r(K,n) )$.
\end{theorem}
Given that the exponent $r$ depends only on the distortion $K$ and $n$, \Cref{thm:pushforward} naturally extends earlier results of Bojarski \cite{Boyar:57} and Gehring \cite{Gehr:73} for planar quasiregular mappings and quasiconformal homeomorphisms in $\mathbb{R}^n$, respectively, and a later result by Meyers and Elcrat for quasiregular mappings in all dimensions \cite{Meyers:Elcrat:75-weakhold}. Regarding the lower bound in \Cref{thm:integrationbyparts}, it is not clear to us if the theorem holds with $q( K, n, 1 )$ or if the dependence on $\|\omega\|_{ \ell^1 }$ can be improved.

Aside from \Cref{thm:remov:bounded:curve}, one of the main motivations for Theorem \ref{thm:integrationbyparts} is the following regularity result for quasiregular curves with minimal distortion.
\begin{corollary}\label{cor:liouville}
Let $m \geq n \geq 2$. Then there is $\epsilon = \epsilon( n, m ) \in (0, n/2)$ such that for every calibration $\omega \in \Lambda^{n}\mathbb{R}^m$, every $f \in W^{1,n-\epsilon}_{loc}(\Omega, \mathbb{R}^m)$ satisfying
\begin{equation*}
    \| Df \|^n
    \leq
    \star f^{\star}\omega
    \quad\text{almost everywhere}
\end{equation*}
has a representative that is $n$-harmonic and $\mathcal{C}^{1}( \Omega )$-regular.
\end{corollary}
We may consider any $\epsilon(n,m) \in ( 0, n-q(1,n,\binom{m}{n}) ) \subset (0,n/2)$ for the constant in \eqref{eq:sharpexponents:intro}. The optimal $\epsilon$ cannot be larger than $n/2$ as follows by post composing the examples of Iwaniec and Martin \cite[Theorem 12.1]{Iwa:Mar:93} with a non-constant linear $1$-quasiregular $\omega$-curve $L \colon \mathbb{R}^n \to \mathbb{R}^m$. See \cite[Theorem 1]{Iwa:Mar:93} for a sharp version of \Cref{cor:liouville} for mappings in even dimension and \cite{Iwa:92} for a similar theorem for mappings in all dimensions.

\Cref{cor:liouville} follows readily from \Cref{thm:integrationbyparts}. Indeed, as $p( 1, n, \|\omega\|_{ \ell^1 } ) > n$, the mapping $f$ has a continuous representative. The continuous representative is then a $\omega$-calibrated curve in the sense of \cite{Hei:Pan:Pry:23} and thus $n$-harmonic and $\mathcal{C}^{1}( \Omega )$-regular by \cite[Theorem 5.1]{Hei:Pan:Pry:23}; notice that the second to last inequality in the proof of \cite[Theorem 5.1]{Hei:Pan:Pry:23} should read
\begin{equation*}
    \frac{ 1 }{ n^{n/2} }
    \int_G \| DF \|^{n}_{HS}
    =
    \int_G F^{\star}\omega
    =
    \int_G v^{\star}\omega
    \leq
    \frac{ 1 }{ n^{n/2} }
    \int_G \|Dv\|^{ n }_{HS}.    
\end{equation*}

\subsection{Removability for quasiregular curves}\label{sec:rem:manifold}
In this section, we connect Sobolev removability to removability problems for quasiregular curves. To motivate the problem, we recall that Ahlfors and Beurling established the following removability theorem \cite{Ahl:Beu:50}:
\begin{theorem}\label{thm:zeroabsolutearea}
Let $E \subset \mathbb{C}$ be compact. Then $E$ is removable for quasiconformal homeomorphisms if and only if $E$ is removable for $W^{1,2}$-functions.
\end{theorem}
In this theorem and the forthcoming discussion, we identify $\mathbb{C}$ with the punctured sphere via the stereographic projection and similarly the $\mathbb{R}^n$ with the punctured sphere $\mathbb{S}^n$. We say that a closed set $E \subset \mathbb{S}^n$ is \emph{removable for quasiconformal homeomorphisms} if every quasiconformal embedding $\phi \colon \mathbb{S}^n \setminus E \rightarrow \mathbb{S}^n$ is the restriction of a quasiconformal homeomorphism $\Phi \colon \mathbb{S}^n \rightarrow \mathbb{S}^n$.

\Cref{thm:zeroabsolutearea} is a modern formulation of \cite[Theorem 9, p. 120]{Ahl:Beu:50}. Indeed, therein Ahlfors and Beurling prove that the \emph{removability for conformal mappings} is equivalent to $E$ being \emph{negligible for extremal length}. In the plane, the equivalence between the removability for conformal mappings and for quasiconformal homeomorphisms follows from Stoïlow factorization theorem, cf. \cite[Section 5.5]{Ast:Iwa:Mar:09}.

Vodop'yanov and Gol'dshtein proved in \cite{Vod:Gol:77} that the class of sets removable for $W^{1,n}$-functions and the class of sets negligible for extremal length coincide and it is not difficult to prove that a set removable for $W^{1,n}$-functions is removable for quasiconformal homeomorphisms. The implication that a set removable for quasiconformal mappings is removable for $W^{1,n}$-functions remains open in higher dimensions; see \cite[Introduction]{Ase:09} or \cite[Section 1.2]{Nta:23:CNED} for recent overviews on the topic. We also highlight the following observation in \cite[p. 291-292]{Kos:99}: if $\mathcal{H}^{n-1}( E ) = 0$, then $E$ is removable for $W^{1,n}$-functions but, in general, such sets can have Hausdorff dimension anywhere in the interval $[0,n]$. Notice that a closed set satisfying the Hausdorff d-measure assumption of \Cref{conj:remov:bounded} or the strict dimension inequality \eqref{eq:hausdorffbound} in \Cref{thm:remov:bounded:curve} is removable for $W^{1,n}$-functions.

We now move to the manifold setting. For this purpose, consider a Riemannian $n$-manifold $M$, a closed set $E \subset M$, and a Riemannian $m$-manifold $N$. We say that $f \in W^{1,n}_{loc}( M \setminus E, N )$ has \emph{finite energy on precompact open subsets of $M$} if $\|Df\|^n$ is integrable in $U \setminus E$ whenever $U \subset M$ is precompact and open.

Consider a closed $\omega \in \Omega^{n}( N )$. If every quasiregular $\omega$-curve $f \colon M \setminus E \to N$ having finite energy on precompact open subsets of $M$ is the restriction of a $K$-quasiregular $\omega$-curve $F \colon M \to N$, we say that $E$ is \emph{removable for quasiregular $\omega$-curves $M \setminus E \to N$ having finite energy on precompact open subsets of $M$}. We prove the following.
\begin{theorem}\label{thm:main:converse:intro}
Let $M$ be an oriented Riemannian $n$-manifold, and let $N$ be a Riemannian $m$-manifold with sectional curvature bounded from above and injectivity radius bounded from below, and $m \geq n \geq 2$. Suppose that $\omega \in \Omega^{n}( N )$ is strongly nonvanishing, closed, and bounded.

Then a closed set $E \subset M$ removable for $W^{1,n}$-functions is removable for quasiregular $\omega$-curves $M \setminus E \rightarrow N$ having finite energy on precompact open subsets of $M$. 

Moreover, if the Hausdorff dimension of $E$ satisfies the strict inequality \eqref{eq:hausdorffbound} in \Cref{thm:remov:bounded:curve} for $K \geq 1$, a continuous extension of a $K$-quasiregular $\omega$-curve $f \colon M \setminus E \to N$ is a $K$-quasiregular $\omega$-curve.
\end{theorem}
In particular, under the Hausdorff dimension bound and the above assumptions on $\omega$, a $K$-quasiregular $\omega$-curve $M \setminus E \to N$ extends to a $K$-quasiregular $\omega$-curve $F \colon M \to N$ if and only if it has a continuous extension. We also note that a closed set of Hausdorff dimension zero, e.g. a discrete set, satisfies the strict inequality \eqref{eq:hausdorffbound} for every $K \geq 1$ and $m \geq n$. With these observations in mind, \Cref{thm:main:converse:intro} can be considered an extension of Riemann's theorem on removability of singularities for holomorphic mappings.

\Cref{thm:main:converse:intro} readily implies the following theorem.
\begin{theorem}\label{cor:quasiregular}
Let $M$ be an oriented Riemannian $n$-manifold, and let $N$ be an oriented Riemannian $n$-manifold with sectional curvature bounded from above and injectivity radius bounded from below, and $n \geq 2$. Then a closed set $E \subset M$ removable for $W^{1,n}$-functions is removable for quasiregular mappings $M \setminus E \rightarrow N$ having finite energy on precompact open subsets of $M$.
\end{theorem}
\Cref{cor:quasiregular} implies that \Cref{thm:main:converse:intro} is qualitatively sharp. Indeed, by \cite[Theorem 12.10]{Iwa:Mar:93} (see also \cite{Rick:95}), there are compact sets $E \subset \mathbb{S}^n$ of arbitrarily small Hausdorff dimension and quasiregular mappings $f \colon \mathbb{S}^n \setminus E \rightarrow \mathbb{S}^n$ that do not extend continuously to any point of $E$. The set $E$ can be taken to be a point or two points, for example.

\Cref{thm:main:converse:intro} also extends to the setting of pseudoholomorphic curves. To this end, we consider a sympletic manifold $N$ endowed with an almost complex structure $J$ tamed by the symplectic form $\omega_{sym}$. A Riemannian structure on such a manifold is \emph{compatible} if the Riemannian tensor is the one canonically defined by the pair $( \omega_{sym}, J )$; see \cite[Section 2.1 and equation (2.1.1)]{McD:Sal:04} for the precise definitions.
\begin{theorem}\label{cor:pseudoholomorphic}
Let $M$ be an oriented Riemannian $2$-manifold, and let $N$ be a symplectic manifold with an almost complex structure $J$ such that the compatible Riemannian structure has its sectional curvature bounded from above and injectivity radius bounded from below. Then a closed set $E \subset M$ removable for $W^{1,2}$-functions is removable for pseudoholomorphic curves $M \setminus E \rightarrow N$ having finite energy on precompact open subsets of $M$.
\end{theorem}
The classes of pseudoholomorphic curves $M \rightarrow N$ and the $1$-quasiregular $\omega_{sym}$-curves coincide in the setting of \Cref{cor:pseudoholomorphic} by standard regularity theory for non-linear Cauchy--Riemann equations \cite[Theorem B.4.1]{McD:Sal:04}, so \Cref{cor:pseudoholomorphic} is immediate from \Cref{thm:main:converse:intro}. See \cite[Section 4.5]{McD:Sal:04} and references therein for related removability results for pseudoholomorphic curves. Similar removability statement can be formulated for the Smith maps \cite{Smi:11,Che:Kari:Mad:20}.

The removability of closed discrete sets is important in the study of the compactified moduli spaces of pseudoholomorphic curves, see e.g. the Gromov compactness theorem \cite{Grom:85,Ye:94} and the Gromov--Witten invariants \cite{Wit:91,Rua:Tia:94,McD:Sal:94} for pseudoholomorphic curves. \Cref{thm:main:converse:intro} is one of the early technical tools needed for defining such invariants for higher dimensional curves.

\subsection{Geometric and topological consequences}\label{sec:geo:top:consi}

To motivate our subsequent application of the removability theorem, \Cref{thm:main:converse:intro}, we recall a growth result for quasiregular mappings by Bonk and Heinonen, cf. \cite[Theorem 1.11]{Bo:He:01}: 
\begin{proposition}\label{prop:Bo:He:weak}
If $f \colon \mathbb{R}^n \rightarrow N$ is a non-constant quasiregular mapping and $N$ is a closed, connected, and oriented Riemannian $n$-manifold that is not a homology sphere, then $f$ has \emph{infinite energy}.
\end{proposition}
\Cref{prop:Bo:He:weak} was recently used by Heikkilä and Pankka \cite{Hei:Pan:23} (see also \cite{Pry:19}) to classify closed, connected, simply connected and oriented Riemannian $4$-manifolds admitting a non-constant quasiregular mapping from $\mathbb{R}^4$. \Cref{thm:main:converse:intro} leads to an analog of \Cref{prop:Bo:He:weak} for quasiregular curves under a further assumption on the form.
\begin{theorem}\label{cor:kunnethideal}
Let $N$ be an oriented Riemannian $n$-manifold with sectional curvature bounded from above and injectivity radius bounded from below, $m \geq n \geq 2$, and that $E \subset \mathbb{S}^n$ is a closed set removable for $W^{1,n}$-functions.

Suppose that $\omega \in \Omega^{n}( N )$ is strongly nonvanishing, closed, bounded, and $[ \omega ] \in \mathcal{K}^n_b(N)$, and $f \colon \mathbb{S}^n \setminus E \rightarrow N$ is a non-constant $K$-quasiregular $\omega$-curve. Then $f$ has infinite energy.
\end{theorem}

A bounded and closed $n$-form $\omega$ induces an equivalence class in the \emph{bounded de Rham cohomology} $H^{n}_{b}( N )$ and we say that $[\omega]$ is in the $n$'th layer $\mathcal{K}^{n}_b(N)$ of the \emph{Künneth ideal} $\mathcal{K}^\star_b( N )$ if
\begin{equation*}
    \omega = d\eta + \sum_{ i = 1 }^{ l } \alpha_i \wedge \beta_i,
\end{equation*}
where $\eta$ and $d\eta$ are bounded and the pair $( \alpha_i, \beta_i)$ consists of bounded and closed $n_i$- and $(n-n_i)$-forms, respectively, for $1 \leq n_i \leq n-1$ and $i = 1,2,\dots,l$ for some $l \in \mathbb{N}$. The nomenclature follows \cite[p. 3]{Heik:23} and is motivated by the Künneth theorem. Observe that if $N$ is $n$-dimensional, oriented, and closed, then the cohomology class of the Riemannian volume form of $N$ is in $\mathcal{K}^{n}_{b}( N )$ if and only if $N$ is not the homology sphere.

A topological perspective on \Cref{cor:kunnethideal} is provided by the following reformulation.
\begin{corollary}\label{cor:kunnethideal:topological}
Let $N$ be an oriented Riemannian $n$-manifold with sectional curvature bounded from above and injectivity radius bounded from below, $2 \leq n \leq m$, and that $E \subset \mathbb{S}^n$ is a closed set removable for $W^{1,n}$-functions.

Suppose that $\omega \in \Omega^{n}( N )$ is strongly nonvanishing, closed, and bounded, and $f \colon \mathbb{S}^n \setminus E \rightarrow N$ is a non-constant $K$-quasiregular $\omega$-curve of finite energy. Then $[ \omega ] \in H^{n}_{b}( N ) \setminus K^{n}_{b}( N )$.
\end{corollary}

We have the following consequence of \Cref{cor:kunnethideal:topological}, removing the infinite-energy assumption from \cite[Theorem 1.2]{Heik:23}.
\begin{corollary}\label{cor:stereographic}
Let $N$ be an oriented and closed Riemannian $m$-manifold, and $m \geq n \geq 2$. Suppose that $\omega \in \Omega^{n}( N )$ is strongly nonvanishing, closed, bounded, and $[ \omega ] \in \mathcal{K}^n_b(N)$, and $f \colon \mathbb{R}^n \rightarrow N$ is a non-constant $K$-quasiregular $\omega$-curve. Then $f$ has infinite energy.
\end{corollary}
\Cref{cor:stereographic} follows from \Cref{cor:kunnethideal:topological} by precomposing the mapping $f$ with the stereographic projection; observe that \Cref{cor:kunnethideal} extends \Cref{prop:Bo:He:weak}.

The proofs of the main theorems can be found in \Cref{sec:mainresults} while Sections \ref{sec:caccio}, \ref{sec:removable}, and \ref{sec:iso} consist of the necessary preparatory work.

\section{Caccioppoli inequality}\label{sec:caccio}
Consider $K \geq 1$ and a calibration $\omega \in \Lambda^n \mathbb{R}^m$. Define $1 \leq q_\omega < p_\omega < \infty$ as the minima and maxima, respectively, for which
\begin{equation*}
    K
    \lambda(n)
    \|\omega\|_{ \ell^1 }
    \left| 1 - \frac{n}{p} \right|
    <
    1.
\end{equation*}
More explicitly, we have
\begin{align}\label{eq:sharpexponents}
    q_\omega =  \frac{ K \lambda(n) \| \omega \|_{ \ell^1 }  }{ K \lambda(n) \|\omega\|_{ \ell^1 } + 1 }n
    \quad\text{and}\quad
    p_\omega = \frac{ K \lambda(n) \|\omega\|_{ \ell^1 } }{ K \lambda(n) \|\omega\|_{ \ell^1 } - 1 }n.
\end{align}
The exponents in \eqref{eq:sharpexponents} are the same as in \eqref{eq:sharpexponents:intro} but we use the above simplified notation in this section. The exponents appear in the following proposition.
\begin{proposition}[Caccioppoli inequality]\label{prop:caccioppoli:estimate}
Consider $K \geq 1$, a calibration $\omega \in \Lambda^{n} \mathbb{R}^m$ and $s \in ( q_\omega, p_\omega )$. Suppose that $h \in W^{1,s}_{loc}( \Omega, \mathbb{R}^m )$ satisfies
\begin{equation*}
    \|Dh\|^n
    \leq
    K \star h^{\star}\omega
    \quad\text{almost everywhere}.
\end{equation*}
Then, for every nonnegative $\phi \in \mathcal{C}^{\infty}_c( \Omega )$ and $c \in \mathbb{R}$, 
\begin{align*}
    \| \phi D(h-c) \|_{ L^{s}( \Omega, \mathbb{R}^{ m \times n } ) }^s
    &\leq
    C
    \| (h-c) \otimes D\phi \|_{ L^{s}( \Omega, \mathbb{R}^{ m \times n } ) }^s
\end{align*}
for a constant $C$ depending only on $s, n, K$ and $\|\omega\|_{ \ell^1 }$; the constant have a singularity at the end points $s \rightarrow q_\omega$ and $s \rightarrow p_\omega$ but remain bounded in compact subintervals of $( q_\omega, p_\omega )$.
\end{proposition}
We postpone the proof of \Cref{prop:caccioppoli:estimate} until the end of the section as we need to establish preliminary results.
\begin{proposition}[Inhomogeneous distortion inequality]\label{prop:inhomogeneous:distortion:modified}
Consider $K \geq 1$, a calibration $\omega \in \Lambda^{n} \mathbb{R}^m$ and $p \in ( 1, \infty )$. Suppose that $h \in W^{1,p}_{loc}( \Omega, \mathbb{R}^m )$ satisfies
\begin{equation*}
    \|Dh\|^n
    \leq
    K \star h^{\star}\omega
    \quad\text{almost everywhere}.
\end{equation*}
If $\phi \in C^{1}( \Omega )$ is nonnegative and $H = \phi h$, then $H$ satisfies the inhomogeneous distortion inequality
\begin{align*}
    \| DH \|^{p}
    \leq
    K \| D H \|^{p-n}\star H^{\star}\omega
    +
    G
    \quad\text{almost everywhere},
\end{align*}
where
\begin{align}\label{eq:inhom:2}
    |G|
    \leq
    &KC(p,n)
    ( \| DH \| + \|h \otimes D\phi\| )^{p-1}
    \| h \otimes D\phi \|
\end{align}
for the constant $C(p,n) =  p + n+|p-n|$.
\end{proposition}

\begin{proof}
We only establish the case $p \neq n$ as $p = n$ is simpler and follows from the argument below with minor modifications.

We observe that
\begin{equation*}
    DH = \phi Dh + h \otimes D\phi.
\end{equation*}
Therefore, setting
\begin{equation*}
    M_\lambda = \lambda DH
    + (1-\lambda) ( DH - h \otimes D\phi ),
\end{equation*}
we have that
\begin{align*}
    \beta( \lambda )
    \coloneqq
    \star M_\lambda^{\star}\omega
\end{align*}
is an affine function since $M_1 - M_0 = h \otimes D\phi$ is either rank one or zero. In fact, the identities
\begin{align*}
    \beta( \lambda )
    &=
    \beta(1)
    +
    \sum_{ j = 1 }^{ n }
    (-1)^{j-1} ( 1 - \lambda )
    \omega\left(
        h \otimes D\phi( e_j ) \wedge DH e_1 \wedge \dots \wedge \widehat{ DH e_j } \wedge \dots \wedge DH e_n
    \right)
    \\
    &=
    \beta(1)
    +
    \sum_{ j = 1 }^{ n }
    (-1)^{j-1}
    (1-\lambda)
    \omega\left(
        h \otimes D\phi( e_j ) \wedge M_{\lambda} e_1 \wedge \dots \wedge \widehat{ M_\lambda e_j } \wedge \dots \wedge M_{\lambda} e_n
    \right),
\end{align*}
and
\begin{align*}
    \beta'( \lambda )
    &=
    \sum_{ j = 1 }^{ n }
    (-1)^{j}
    \omega\left(
        h \otimes D\phi( e_j ) \wedge DH e_1 \wedge \dots \wedge \widehat{ DH e_j } \wedge \dots \wedge DH e_n
    \right)
    \\
    &=
    \sum_{ j = 1 }^{ n }
    (-1)^{j}
    \omega\left(
        h \otimes D\phi( e_j ) \wedge M_\lambda e_1 \wedge \dots \wedge \widehat{ M_\lambda e_j } \wedge \dots \wedge M_{\lambda} e_n
    \right)
\end{align*}
hold where $\widehat{v}$ means that the vector is omitted from the wedge product. Then Hadamard's inequality implies
\begin{align*}
    | \beta(\lambda) |
    \leq
    \|M_{\lambda}\|^{n}
    \quad\text{and}\quad
    | \beta'(\lambda) |
    \leq
    n \| h \otimes D\phi \| \| M_\lambda \|^{n-1}.
\end{align*}
We define the function
\begin{equation*}
    G( \lambda )
    =
    \| M_\lambda \|^{ p }
    -
    K \| M_\lambda \|^{ p - n } \beta( \lambda ),
\end{equation*}
where it is understood that if $\|M_\lambda\| = 0$, then $\| M_\lambda \|^{ p - n } \beta( \lambda ) = 0$. We only care about the points where $G(0) \leq 0$, i.e. almost every point of $\Omega$. We fix such an $x \in \Omega$.

We consider the cases where $\| M_\lambda \| = 0$ for some $0 \leq \lambda \leq 1$ and $\| M_\lambda \| > 0$ for every $0 \leq \lambda \leq 1$ separately.

It is clear that if $\| M_\lambda \| = 0$ for some $0 \leq \lambda < 1$, then $DH = (1-\lambda) h \otimes D\phi$ and we have $M_\mu = ( \mu - \lambda ) h \otimes D\phi$ for $0 \leq \mu \leq 1$. So $M_\mu$ has rank at most one everywhere, yielding $\beta \equiv 0$. Hence
\begin{equation*}
    G( 1 ) = \| M_1 \|^{p} = ( 1-\lambda)^p \| h \otimes D\phi \|^p,
\end{equation*}
implying the desired estimate \eqref{eq:inhom:2}. On the other hand, if $\| M_1 \| = 0$, then $G(1) = 0$ and \eqref{eq:inhom:2} clearly holds.

To finish the proof of \eqref{eq:inhom:2}, it remains to consider the case when $\| M_\lambda \| > 0$ for every $0 \leq \lambda \leq 1$. As $G( 0 ) \leq 0$, it suffices to prove
\begin{equation}\label{eq:growthestimate:F}
    | G'(\lambda) |
    \leq
    K
    C(p,n)
    \left( \| M_1 \| + \|M_1-M_0\| \right)^{p-1}
    \| M_1 - M_0 \|
\end{equation}
for almost every $0 < \lambda < 1$ and integrate.

Denoting $\alpha_r( \lambda ) = \| M_\lambda \|^r$ for $r \in \mathbb{R} \setminus \left\{0\right\}$, we have $\alpha_r'( \lambda ) = r ( \alpha_1( \lambda ) )' \alpha_{r-1}( \lambda )$ for almost every $0 < \lambda < 1$ where $| \alpha_1'( \lambda ) | \leq \| M_1 - M_0 \|$ at every point of differentiability of $\alpha_1$. Then, for almost every $0 < \lambda < 1$,
\begin{align*}
    | G'( \lambda ) |
    &\leq
    p
    \|M_1 - M_0\| \alpha_{p-1}(\lambda)
    +
    K
    \alpha_{p-n}( \lambda ) | \beta'(\lambda) |
    \\
    &+
    K
    |p-n| \| M_1-M_0 \| \alpha_{p-n-1}( \lambda ) |\beta( \lambda )|.
\end{align*}
Therefore, by using the inequalities for $|\beta( \lambda )|$ and $|\beta'(\lambda)|$ and recalling that $h\otimes D\phi = M_1 - M_0$, we deduce that
\begin{equation}\label{eq:growthestimate:Fprime}
    |G'(\lambda)|
    \leq
    \left(
        p + K( n+|p-n| )
    \right)
    \|M_1-M_0\| \alpha_{p-1}( \lambda ).
\end{equation}
To obtain \eqref{eq:growthestimate:F} from \eqref{eq:growthestimate:Fprime}, we apply the inequality $\alpha_{p-1}( \lambda ) \leq ( \| M_1 \| + ( \|M_1-M_0\| ) )^{p-1}$.
\end{proof}

We also need the following lemma for the proof of \Cref{prop:caccioppoli:estimate}.
\begin{lemma}\label{lemm:integrationbyparts}
Suppose that $n+1 \leq 2p \leq 3n$, $h \in W^{1,p}( \Omega, \mathbb{R}^m )$ and at least $(m-n)+1$ components of $h$ vanish at the boundary of $\Omega$ in the Sobolev sense. Then
\begin{align}\label{eq:integrationbyparts:1}
    \left|
        \int_\Omega
            \sum_{ I } \omega_I
            \frac{ h^{\star} dx^{I} }{ | dh^{i_1} |^{ \frac{n-p}{n} } \dots | dh^{ i_n } |^{ \frac{n-p}{n} } }
    \right|
    \leq
    C(n)
    \|\omega\|_{ \ell^1 }
    \left| \frac{n-p}{n} \right|
    \int_{\Omega} \|Dh\|^p\,d\mathcal{H}^n
\end{align}
and
\begin{align*}
    \left|
        \int_\Omega \|Dh\|^{p-n} h^{\star}\omega
        -
        \int_\Omega
            \sum_{ I } \omega_I
            \frac{ h^{\star} dx^{I} }{ | dh^{i_1} |^{ \frac{n-p}{n} } \dots | dh^{ i_n } |^{ \frac{n-p}{n} } }
    \right|
    &\leq
    \|\omega\|_{ \ell^1 }
    \left| \frac{n-p}{n} \right|
    \int_\Omega \|Dh\|^p \,d\mathcal{H}^n
\end{align*}
for every calibration $\omega \in \Lambda^{n}\mathbb{R}^m$.
\end{lemma}
\begin{proof}
We argue exactly as in the proof of the fundamental inequality for Jacobians in \cite[p. 352]{Iw:Ma:01}. The key observation is that $\star h^{\star} dx^{I}$ is the Jacobian of $h_I = \pi_I \circ h$ with coordinate projection $\pi_I(x) = ( x^{i_1}, \dots, x^{i_n} )$.  Notice that at least one of the components of $h_I$ vanishes at the boundary of $\Omega$, so \cite[Theorem 13.6.1]{Iw:Ma:01} is applicable for $h_I$ and its components exactly as in \cite[p. 352]{Iw:Ma:01}. Indeed, $h_I$ plays the role of $f$ in the argument. In inequality (13.68) in \cite[p. 352]{Iw:Ma:01}, we estimate $\|Dh_I\|$ from above by $\|Dh\|$. More explicitly, we have
\begin{align*}
    \left|
        \int_\Omega
            \frac{ h^{\star} dx^{I} }{ | dh^{i_1} |^{ \frac{n-p}{n} } \dots | dh^{ i_n } |^{ \frac{n-p}{n} } }
    \right|
    \leq
    C(n)\left| \frac{n-p}{n} \right|
    \int_\Omega \|Dh\|^p \,d\mathcal{H}^n.
\end{align*}
Then \eqref{eq:integrationbyparts:1} readily follows by triangle inequality. Similarly, we need only replace the operator norm $\| Dh_I \|$ by the larger quantity $\| Dh \|$ in order to derive the inequality
\begin{align*}
    \left|
        \frac{ | dh^{i_1} | \dots | dh^{ i_n } | }{ \|Dh\|^n }
        -
        \left( \frac{ | dh^{i_1} | \dots | dh^{ i_n } | }{ \|Dh\|^n } \right)^{ \frac{p}{n} }
    \right|
    \leq
    \left| \frac{n-p}{n} \right|
\end{align*}
which then implies
\begin{align*}
    \left|
        \frac{ h^{\star} dx^{I} }{ \|Dh\|^{ n-p } }
        -
        \frac{ h^{\star} dx^{I} }{ | dh^{i_1} | \dots | dh^{ i_n } |^{ \frac{n-p}{n} } }
    \right|
    &\leq
    \left|
        \frac{ | dh^{i_1} | \dots | dh^{ i_n } | }{ \|Dh\|^{n-p} }
        -
        \left( | dh^{i_1} | \dots | dh^{ i_n } | \right)^{ \frac{p}{n} }
    \right|
    \\
    &=
    \left|
        \frac{ | dh^{i_1} | \dots | dh^{ i_n } | }{ \|Dh\|^{n} }
        -
        \left( \frac{ | dh^{i_1} | \dots | dh^{ i_n } | }{ \|Dh\|^{n} } \right)^{ \frac{p}{n} }
    \right|
    \|Dh\|^{ p }
    \\
    &\leq
    \left| \frac{ n-p}{n} \right|
    \|Dh\|^{ p }.
\end{align*}
similarly to \cite[p. 352]{Iw:Ma:01}. Now triangle inequality gives
\begin{align*}
    \left|
        \sum_{I}
        \frac{ \omega_I h^{\star} dx^{I} }{ \|Dh\|^{ n-p  } }
        -
        \frac{ \omega_I h^{\star} dx^{I} }{ | dh^{i_1} |^{ \frac{n-p}{n} } \dots | dh^{ i_n } |^{ \frac{n-p}{n} } }
    \right|
    \leq
    \| \omega \|_{ \ell^1 }
    \left| \frac{n-p}{n} \right|
    \| Dh \|^{ p }.
\end{align*}
So integrating and applying the triangle inequality finishes the proof.
\end{proof}

We now prove a fundamental inequality for calibrations similar to the fundamental inequality of volume forms (Jacobians), cf. \cite[Theorem 13.7.1]{Iw:Ma:01}.
\begin{corollary}[Fundamental inequality for calibrations]\label{cor:fundamentalinequality}
Suppose that $p \in (1,\infty)$, $h \in W^{1,p}( \Omega, \mathbb{R}^m )$ and at least $(m-n)+1$ components of $h$ vanish at the boundary of $\Omega$ in the Sobolev sense. Then
\begin{align*}
    \left|
        \int_\Omega \|Dh\|^{p-n} h^{\star}\omega
    \right|
    \leq
    \lambda(n)
    \|\omega\|_{ \ell^1 }
    \left| 1 - \frac{n}{p} \right|
    \int_\Omega \|Dh\|^p \,d\mathcal{H}^n
\end{align*}
for every calibration $\omega \in \lambda^{n}\mathbb{R}^m$.
\end{corollary}
\begin{proof}
We set $\lambda(n) = 3( C(n) + 1 )$ for the constant $C(n)$ in \Cref{lemm:integrationbyparts}. Then the claim is immediate from \Cref{lemm:integrationbyparts} in the range $n+1 \leq 2 p \leq 3n$ after we observe that $|(n-p)/n| \leq (p/n)|(p-n)/p| \leq 3|(p-n)/p|$.

If $2p > 3n$ or $1 \leq p < (n+1)/2$, then $\lambda(n)\left|1-n/p\right| > 1$ so a better inequality is obtained using the trivial one: $\|Dh\|^{p-n} \star h^{\star}\omega \leq \|Dh\|^p$. Either way, the claim holds.
\end{proof}

\begin{remark}
When $\|\omega\|_{ \ell^1 } = 1$, we obtain the same constant as in \cite[Theorem 13.7.1]{Iw:Ma:01}. It is not clear to us if the factor $\|\omega\|_{ \ell^1 }$ is truly necessary in \Cref{cor:fundamentalinequality}. Whether or not $\|\omega\|_{\ell^1}$ is needed in \Cref{cor:fundamentalinequality} is directly related to the size of critical interval in \Cref{prop:caccioppoli:estimate}.
\end{remark}

\begin{proof}[Proof of \Cref{prop:caccioppoli:estimate}]
We first derive the claim in the case $c = 0$ since the general case follows from the argument below simply by considering $h-c$ in place of $h$.

Here $h$ satisfies the assumptions of \Cref{prop:inhomogeneous:distortion:modified}; consider next a nonnegative $\phi \in \mathcal{C}^{1}_{c}( \Omega )$. Then $H = \phi h$ satisfies an inhomogeneous distortion inequality with the inhomogeneous term satisfying the pointwise inequality \eqref{eq:inhom:2}. Integrating the inhomogeneous distortion inequality gives
\begin{equation*}
    \int_\Omega \|DH\|^p \,d\mathcal{H}^n
    \leq
    K
    \int_\Omega \|DH\|^{p-n} H^{\star}\omega
    +
    \int_\Omega |G|
    \,d\mathcal{H}^n.
\end{equation*}
\Cref{lemm:integrationbyparts} implies that
\begin{equation*}
    \int_\Omega \|DH\|^{p-n} H^{\star}\omega
    \leq
    \lambda(n) \|\omega\|_{ \ell^1 }
    \left| 1 - \frac{n}{p} \right|
    \int_\Omega \|DH\|^{p}
    \,d\mathcal{H}^n.
\end{equation*}
By combining the inequalities and using the pointwise upper bound \eqref{eq:inhom:2} for $G$ in \Cref{prop:inhomogeneous:distortion:modified}, together with a rearrangement of the terms, we obtain that
\begin{align*}
    \int_\Omega \|DH\|^p \,d\mathcal{H}^n
    &\leq
    \frac{ KC }{ 1 - K \lambda(n) \|\omega\|_{ \ell^1 }
    \left| 1 - \frac{n}{p} \right| }
    \int_\Omega 
    ( \| DH \| + \|h \otimes D\phi\| )^{p-1}
    \| h \otimes D\phi \|
    \,d\mathcal{H}^n.
\end{align*}
We add the integral of $\|h \otimes D\phi\|^p$ to both sides and use the convexity of $t \mapsto t^p$ and $\| h \otimes D\phi \|^p \leq ( \|Dh\| + \|h \otimes D\phi\| )^{p-1} \|h \otimes D\phi\|$ to get
\begin{align*}
    &\int_\Omega \left( \| D H  \| + \| h \otimes D\phi \| \right)^p \,d\mathcal{H}^n
    \\
    &\leq
    C(p,n,K,\|\omega\|_1)
    \int_\Omega 
    ( \| DH \| + \|h \otimes D\phi\| )^{p-1}
    \| h \otimes D\phi \|
    \,d\mathcal{H}^n.
\end{align*}
Then Hölder's inequality yields that
\begin{align*}
    \| \| D H  \| + \| h \otimes D\phi \| \|_{ L^{p}( \Omega, \mathbb{R}^{ m \times n } ) }^p
    \leq
    C
    \| h \otimes D\phi \|_{ L^{p}(\Omega, \mathbb{R}^{ m \times n }) }^p.
\end{align*}
Recalling that $\phi Dh = DH - h \otimes D\phi$, triangle inequality gives that
\begin{align*}
    \| \phi Dh \|_{ L^{p}( \Omega, \mathbb{R}^{ m \times n } ) }^p
    &\leq
    \| \|DH\| + \| DH - h \otimes D\phi \| \|_{ L^{p}( \Omega, \mathbb{R}^{ m \times n } ) }^p
    \\
    &\leq
    C
    \| h \otimes D\phi \|_{ L^{p}(\Omega, \mathbb{R}^{ m \times n }) }^p
\end{align*}
as claimed.
\end{proof}

\section{Removable sets}\label{sec:removable}
We consider a notion of removability for functions and mappings.
\begin{definition}\label{def:removability}
Let $M$ be a Riemannian $n$-manifold, $E \subset M$ a closed set and $p \in [1,\infty)$. We say that
\begin{enumerate}
    \item $E$ is \emph{removable for $W^{1,p}$-functions} if $\mathcal{H}^{n}( E ) = 0$ and every $u \in W^{1,p}( M \setminus E)$ is the restriction of some $h \in W^{1,p}( M )$.
    \item $E$ is \emph{removable for $W^{1,p}$-mappings} if $\mathcal{H}^{n}( E ) = 0$ and for every complete metric space $X$, every $u \in W^{1,n}( M \setminus E, X)$ is the restriction of some $h \in W^{1,n}( M, X )$.
\end{enumerate}
\end{definition}
Here $\mathcal{H}^n$ is the $n$-dimensional Hausdorff measure. The measure is normalized so that $\mathcal{H}^n$ coincides with the Riemannian volume measure on $M$.

The metric-valued Sobolev theory we use is the upper gradient approach due to Heinonen, Koskela, Shanmugalingam and Tyson, cf. \cite{Hei:Kos:Sha:Ty:01} and \cite[p. 181]{Hei:Kos:Sha:Ty:15}. See \cite[Section 10]{Hei:Kos:Sha:Ty:15} for an explanation on how the upper gradient approach relates to other definitions. We emphasize that when the target is a Euclidean space, the definition coincides with the standard approach, cf. \cite[Theorem 6.1.17]{Hei:Kos:Sha:Ty:15}. The removability notions in \Cref{def:removability} are equivalent; the clearest reference implying this is \cite{GB:Iko:Zhu:22} but we include the details for the convenience of the reader.
\begin{proposition}\label{prop:removability:mappings}
Consider $p \in [1,\infty)$, a Riemannian $n$-manifold $M$, and a closed set $E \subset M$. Then $E$ is removable for $W^{1,p}$-functions if and only if $E$ is removable for $W^{1,p}$-mappings.
\end{proposition}
\begin{proof}
We only prove the 'only if'-direction as the 'if'-direction is immediate from the definitions. Consider a complete metric space $X$ and $u \in W^{1,n}( M \setminus E, X )$. By definition, there is $v \in W^{1,n}(  M \setminus, L^{\infty}(X) )$ and an isometric embedding
\begin{equation*}
    \iota \colon X \xhookrightarrow{} L^{\infty}(X)
\end{equation*}
for which $v = \iota \circ u$ and $\| u \|_{ W^{1,p}( M \setminus E, X ) } = \| v \|_{ W^{1,p}( M \setminus E, L^\infty(X) ) }$; $L^{\infty}(X)$ is given the supremum norm.

By the inner regularity of $\mathcal{H}^n$ on $M \setminus E$ and Lusin--Egoroff theorem, there is an increasing sequence of compact sets $( K_i )_{ i = 1 }^{ \infty }$ in $M \setminus E$ for which
\begin{equation*}
    v|_{ K_i } \text{ is continuous}
    \quad\text{and}\quad
    \int_{ ( M \setminus E ) \setminus K_i } | v |^p + \| Dv \|^p \,d\mathcal{H}^n \leq 2^{-(i+3)p};
\end{equation*}
here $\| Dv \|$ is the minimal $p$-weak upper gradient of $v$.

Since $L^{\infty}(X)$ has the metric approximation property, see e.g. \cite[p. 258, (9)]{Ko:79}, and $F_i = v( K_i )$ is compact, there is a finite rank linear map $T_i \colon L^{\infty}( X ) \rightarrow L^{\infty}( X )$ such that
\begin{equation*}
    \| T_i \| \leq 1
    \quad\text{and}\quad
    \sup_{ w \in F_i } \| w - T_i(w) \| \leq \frac{ 2^{-(i+3)} }{ \left(1 + \mathcal{H}^n( K_i ) \right)^{1/p} }.
\end{equation*}
Then $v_i \coloneqq T_i( v ) \in W^{1,p}( M \setminus E, L^{\infty}(X) )$ and
\begin{equation}\label{eq:lowersemicontinuity:ofenergy}
    | v - v_i |(x)
    \leq
    \left( \frac{ 2^{ -(i+3) } }{ \left(1 + \mathcal{H}^n( K_i ) \right)^{1/p} }
    \chi_{ K_i }
    +
    2 | v |
    \chi_{ ( M \setminus E ) \setminus K_i } \right)(x),
    \quad\text{for $x \in M \setminus E$.}
\end{equation}
Here $\chi_{A}$ is the characteristic function of the set $A$.

Since the range of $T_i$ has dimension $N_i \in \mathbb{N}$, there exists $h_i \in W^{1,p}( M, L^{\infty}(X) )$ extending $v_i$. Indeed, there is a bounded linear map $L_i \colon L^{\infty}( X ) \rightarrow \mathbb{R}^{ N_i }$ so that the restriction of $L_i$ to the range of $T_i$ is bi-Lipschitz. Then $L_i \circ v_i$ can be extended componentwise and the existence of an extension $h_i \in W^{1,n}( M, L^{\infty}(X) )$ of $v_i$ follows from $L_i$ being bi-Lipschitz in the range of $T_i$. Since $\mathcal{H}^n(E) = 0$ and $T_i$ has operator norm one, we deduce that
\begin{align*}
    &\| h_i \|_{ W^{1,p}( M, L^{\infty}(X) ) }
    =
    \| v_i \|_{ W^{1,p}( M \setminus E, L^{\infty}(X) ) }
    \leq
    \| v \|_{ W^{1,p}( M \setminus E, L^{\infty}(X) ) }
    \quad\text{and}    
    \\
    &\| h_{i+1} - h_i \|_{ L^{p}( M, L^{\infty}(X) ) }
    =
    \| v_{i+1} - v_{i} \|_{ L^{p}( M \setminus E, L^{\infty}(X) ) }
    \leq
    2^{-i}
    \quad\text{for every $i \in \mathbb{N}$,}
\end{align*}
where the $L^{p}$-norm estimate for $v_{i+1}-v_i$ follows from \eqref{eq:lowersemicontinuity:ofenergy}.

We conclude that $( h_i )_{ i = 1 }^{ \infty }$ is bounded in $W^{1,p}( M, L^{\infty}(X) )$ and Cauchy in the space $L^{p}( M, L^{\infty}(X) )$; note that in case $p = 1$, $\|Dh_i\|$ is bounded from above by the zero extension of $\|Dv\|$, so $( \|Dh_i\| )_{ i = 1 }^{ \infty }$ has a weakly convergent subsequence by Dunford--Pettis theorem if $p = 1$ and by reflexivity if $p > 1$. We conclude that the Cauchy limit of $( h_i )_{ i = 1 }^{ \infty}$ is an extension of $v$ and, by lower semicontinuity of energy, a representative of the extension of $v$ has a $p$-integrable weak upper gradient, cf. \cite[Theorem 7.3.9]{Hei:Kos:Sha:Ty:15}. Thus there is an extension $h \in W^{1,p}( M, L^{\infty}(X) )$ of $v$ as claimed.
\end{proof}

\section{Continuity of quasiregular curves and isoperimetric inequalities}\label{sec:iso}
We refer the reader to \cite[Section 6]{Iko:23} for the following terminology about isoperimetric inequalities of Euclidean type. We recall a special case of \cite[Theorem 6.8]{Iko:23}.
\begin{theorem}\label{thm:isoperimetric:to:continuity}
Suppose that $N$ is a complete Riemannian manifold that supports an isoperimetric inequality of Euclidean type with dimension $n$, mass bound $M'$, and constant $A'$, and with dimension $(n-1)$, mass bound $M$, and constant $A$, respectively, for some $n \geq 2$.

Suppose that $\omega \in \Omega^{n}( N )$ is strongly nonvanishing, closed, and bounded, and $F \in W^{1,n}_{loc}( \mathbb{B}^{n}; N )$ satisfies $\|DF\|^n \leq K \star F^{*}\omega$ almost everywhere. Then $F$ has a continuous representative that is a $K$-quasiregular $\omega$-curve. In fact, the continuous representative is locally Hölder continuous with exponent $\alpha = \alpha( K, \mathrm{inf}\|\omega\|, \sup\|\omega\|, A, n) \in (0,1]$.
\end{theorem}
In the statement $\mathbb{B}^{n}$ refers to the closed Euclidean ball.

Our interest in \Cref{thm:isoperimetric:to:continuity} comes from \Cref{prop:sectional:to:contractibility}, stated below, connecting sectional curvature upper bounds and injectivity radius lower bounds to Lipschitz contractiblity in the sense of \cite{Wen:07}. To this end, we say that a bounded set $B \subset N$ is \emph{$(\beta,\gamma)$-Lipschitz contractible} if there exists $x_0 \in X$ and a Lipschitz map $\varphi \colon [0,1] \times B \to X$ satisfying $\varphi(1,x) = x$, $\varphi( 0, x ) = x_0 \in X$ for every $x \in X$ with
\begin{align*}
    d( \varphi(t,x), \varphi(t',x') )
    \leq
    \beta|t-t'| + \gamma d(x,x')
    \quad\text{for $t,t' \in [0,1]$ and $x,x' \in B$.}
\end{align*}
We say that $B$ is \emph{$\gamma$-contractible} if we may take $\beta = \gamma \diam(B)$. It follows from \cite[Proposition 3.4]{Wen:07} that if there are $\delta > 0$ and $\gamma > 0$ such that every set of diameter at most $\delta$ is $\gamma$-contractible, then $N$ admits \emph{Euclidean coning inequalities} up to scale $\delta$ of all dimensions and thus, by \cite[Theorem 5.6]{Wen:07}, $N$ admits \emph{Euclidean isoperimetric inequalities} for all dimensions up to some mass. Lastly, we recall that the \emph{injectivity radius} $r_{\mathrm{inj}}(x)$ at $x \in N$ is the supremum over the radii $r$ at which the exponential map $\mathrm{exp}_x \colon T_xN \cap B(0,r) \to N$ is a diffeomorphism onto its image $B( x, r )$. We say that $N$ has an \emph{injectivity radius lower bound $r_0$ > 0} if $\inf_{x\in N} r_{\mathrm{inj}}(x) \geq r_0$. We recall that a Riemannian manifold with a lower bound on the injectivity radius is complete. We finish this section with the following proposition.
\begin{proposition}\label{prop:sectional:to:contractibility}
Let $N$ be a Riemannian manifold with a sectional curvature upper bound $\kappa \in \mathbb{R}$ and an injectivity radius lower bound $r_0>0$. Consider
\begin{equation*}
    \delta_0
    =
    \left\{
    \begin{split}
        &\min\left\{ r_0, \pi/ ( 2 \sqrt{\kappa} ) \right\}, \quad&&\text{if $\kappa > 0$}
        \\
        &r_0, \quad&&\text{if $\kappa \leq 0$.}
    \end{split}
    \right.
\end{equation*}
If $\delta \in (0,\delta_0/2)$, then every set of diameter at most $\delta$ is $2$-Lipschitz contractible. In particular, $N$ admits Euclidean isoperimetric inequalities for all dimensions up to some mass.
\end{proposition}
\begin{proof}
The argument is fairly standard but we provide an outline of the argument for the convenience of the reader. It follows from \cite[Theorem 1]{Gulliver:75} that balls of radius $r \in (0,\delta_0)$ are convex and if we consider $\delta$ as in the claim, then balls of radius $\delta$ in $N$ are $\mathrm{CAT}( \kappa )$-spaces; for the definition of $\mathrm{CAT}( \kappa )$-spaces, see e.g. \cite[p. 158, Definition 1.1]{Brid:Haef:99}. Indeed, the balls of radius $2\delta$ satisfy the assumptions of \cite[p. 169, Lemma 1A.1]{Brid:Haef:99} by (the proof of) \cite[Proposition 1A.5]{Brid:Haef:99}. The $\mathrm{CAT}( \kappa )$-conclusion for balls of radius $\delta$ then follows from \cite[p. 170, Proposition 1A.2]{Brid:Haef:99}.

Now in case $x_0 \in N$ and $r \in (0,\delta]$, we consider $\varphi \colon [0,1] \times B( x_0, r ) \to B( x_0,r)$ where $\varphi( t, x ) = \gamma_{x_0,x}(t)$ for the constant speed geodesic $\gamma_{x_0,x} \colon [0,1] \to B( x_0, r )$ joining $x_0$ to $x$. Consider the distance $D = d( \varphi(t,x), \varphi(t',x') )$ for $t,t' \in [0,1]$ and $x, x' \in B( x_0, r )$. The claim readily follows if we prove that
\begin{align}\label{eq:desiredupperbound}
    D \leq 2(  r|t'-t| + d( x, x' ) ).
\end{align}
We consider a geodesic triangle $T$ with sides defined by the constant speed geodesics $\gamma_{x_0,x}, \gamma_{x_0,x}, \gamma_{x,x'} \colon [0,1] \to B( x_0, r )$. Given the upper bound on the radius $r$, we find constant speed geodesics $\widehat{\gamma}_{y_0,y}, \widehat{\gamma}_{y_0,y'}, \widehat{\gamma}_{y,y'} \colon [0,1] \to B( y_0, r )$ on the model space $M_{\kappa}$ of constant sectional curvature $\kappa$ such that $( \gamma_{x_0,x}, \widehat{\gamma}_{y_0,y} )$ have equal length, and similarly, the pairs $( \gamma_{x_0,x'}, \widehat{\gamma}_{y_0,y'} )$ and $( \gamma_{x,x'}, \gamma_{y,y'} )$ have equal length, respectively. The definition of $\mathrm{CAT}( \kappa )$-spaces guarantee that if $T'$ is the geodesic triangle formed by these geodesics and $p \colon T' \to T$ is the map sending $\gamma_{y_0,y}(s)$ to $\gamma_{x_0,x}(s)$ for each $s \in [0,1]$, and similarly for the other two geodesics, then $p$ is $1$-Lipschitz. Therefore, in particular,
\begin{align*}
    D
    \leq
    d( \widehat{\gamma}_{y_0,y}(t), \widehat{\gamma}_{y_0,y'}(t')).
\end{align*}
By this inequality and as the side lengths of $T'$ and $T$ coincide, \eqref{eq:desiredupperbound} holds for $N$ if we are able to prove the corresponding claim in the special case that $N$ has constant sectioncal curvature $\kappa$. We assume this from this point onwards. In case $\kappa > 0$, we use normal coordinates centered at $x_0$ in which the Riemannian tensor is of the form $g_\kappa = dr^2 + \frac{1}{\kappa} \sin^2( \sqrt{\kappa} r ) g_{ \mathbb{S} }$ on the Euclidean sphere of radius $r$ where $g_{\mathbb{S}}$ is the Riemannian metric on the Euclidean sphere of radius one, see e.g. \cite[Theorem 10.14]{Lee:18}. Since radial geodesics in normal coordinates are of the form $s \mapsto s v$, we have that $\varphi(t,x) = tx$ and $\varphi(t',x') = t'x'$ in these coordinates. To obtain the desired upper bound for $D$, it suffices to bound from above the $g_\kappa$-length of the Euclidean segments joining $tx$ to $t'x$ and $t'x$ to $t'x'$. The length of the former segment corresponds to the Euclidean length and is thus bounded from above by $|t-t'| r$. The length of the latter segment is also bounded from above by the Euclidean length because $\sin( \sqrt{\kappa}r ) / \sqrt{\kappa} \leq r$ for $r \in [0, \pi / (2\sqrt{\kappa})]$. Since $\frac{r}{2} \leq \sin( \sqrt{\kappa}r ) / \sqrt{\kappa}$ for $r \in [0, \pi / (2\sqrt{\kappa})]$, the Euclidean length from $x$ to $x'$ is bounded from above by $2d(x',x)$. Now \eqref{eq:desiredupperbound} follows. On the other hand, in case $\kappa \leq 0$, we may replace $M_{\kappa}$ by $M_0$ in the construction of $T'$ above and thus it is enough to consider the case $\kappa = 0$. Then \eqref{eq:desiredupperbound} holds with the constant one since $T'$ is isometric to a triangle in $\mathbb{R}^2$ and there the inequality is clear.
\end{proof}

\section{Proofs of the main theorems}\label{sec:mainresults}

\begin{proof}[Proof of \Cref{thm:integrationbyparts}]  
Consider a calibration $\omega \in \Lambda^{n}\mathbb{R}^m$. We consider the critical interval $( q_\omega, p_\omega )$ from \Cref{prop:caccioppoli:estimate}; we recall the notation $q_\omega = q( K, n, \|\omega\|_{ \ell^1 } )$ and $p_\omega = p( K, n, \|\omega\|_{ \ell^1 } )$.

Let $p \in ( q_\omega, p_\omega )$ and suppose that $f \in W^{1,p}_{loc}( \Omega, \mathbb{R}^m )$ satisfies
\begin{equation*}
    \| D f \|^{n} \leq K \star f^{\star}\omega
    \quad\text{almost everywhere}.
\end{equation*}
We apply \Cref{prop:caccioppoli:estimate} for $h = f$ and obtain
\begin{align*}
    \| \phi Df \|_{ L^{p}( \Omega, \mathbb{R}^{ m \times n } ) }
    &\leq
    C
    \| (f-c) \otimes D\phi \|_{ L^{p}( \Omega, \mathbb{R}^{ m \times n } ) }
\end{align*}
for every $c \in \mathbb{R}$ and nonnegative $\phi \in \mathcal{C}^{\infty}_c( \Omega )$. 

Consider a closed cube $Q$ such that $2Q \subset \Omega$, i.e. a closed cube such that the closed cube with twice the side length and same center is contained in $\Omega$.

If either $n = 2$ and $p > 2$ or $n \geq 3$, we denote $\widehat{p} = p$. If $n = 2$ and $p \leq 2$, we denote $\widehat{p} = 2p$. In the last case, the exponent $q = 2p/(1+p) < p < 2$ satisfies $2q/(2-q) = \widehat{p}$ and in the other cases there is $1 < q < \min\left\{n,p\right\}$ such that $nq/(n-q) = \widehat{p} = p$.

We consider $\phi \in \mathcal{C}^{\infty}_c( \Omega )$ with $0 \leq \phi \leq 1$, $\phi|_{Q} \equiv 1$ and $\phi|_{ \Omega \setminus 2Q } \equiv 0$, satisfying $\|D\phi\| \leq 3/\diam( Q )$. 

Let $c$ be the integral average of $f$ over $2Q$. Then, by the choice of $\phi$ and the Caccioppoli inequality, we get that
\begin{align*}
    \| Df \|_{ L^{p}( Q, \mathbb{R}^{ m \times n } ) }
    &\leq
    \| \phi Df \|_{ L^{p}( \Omega, \mathbb{R}^{ m \times n } ) }
    \leq
    C \frac{ 3 }{ \diam (Q) }
    \| f - c \|_{ L^{p}( 2Q, \mathbb{R}^{ m } ) }.
\end{align*}
Taking integral averages over $2Q$ and applying Hölder's inequality gives that
\begin{align*}
    \frac{ 3 }{ \diam(Q) }
    \left( \aint{ 2Q } | f - c |^{ p }  \right)^{ \frac{1}{ p } }
    \leq
    \frac{ 3 }{ \diam(Q) }
    \left( \aint{ 2Q } | f - c |^{ \widehat{p} }  \right)^{ \frac{1}{ \widehat{p} } },
\end{align*}
where $\aint{}$ refers to the integral average. We apply Sobolev--Poincaré inequality for $f$ with the exponent $q$ fixed above. Then the above inequality gives that
\begin{align}\label{eq:sobolevpoincare}
    \left( \aint{ Q } \|Df\|^p \right)^{ \frac{1}{p} }
    \leq
    C(p,n,K,\|\omega\|_1)
    \left( \aint{2Q} \|Df\|^q \right)^{ \frac{1}{q} }
    <
    \infty.
\end{align}
Since $2Q \subset \Omega$ was arbitrary, the Generalized Gehring lemma holds for the function $g = \|Df\|^q$, see e.g. \cite[Corollary 14.3.1]{Iw:Ma:01}, yielding improved integrability for $g$. So $f \in W^{1,s}_{loc}( \Omega, \mathbb{R}^n )$ for some $s > p$. 

The argument above implies that the set of exponents for which $f \in W^{1,s}_{loc}( \Omega, \mathbb{R}^n )$ is relatively open in $(q_\omega, p_\omega)$. The claim follows if we are able to prove that the set of such $s$ is also relatively closed in $( q_\omega, p_\omega )$. However, this follows from a simple continuity argument using the observation that $C(K,n,\|\omega\|_{ \ell^1},p)$ in \eqref{eq:sobolevpoincare} is bounded in every compact subinterval of $( q_\omega, p_\omega )$.
\end{proof}

\begin{proof}[Proof of \Cref{thm:caccioppoli:estimate}]
The claim is immediate from \Cref{prop:caccioppoli:estimate}.
\end{proof}

\begin{proof}[Proof of \Cref{thm:remov:bounded:curve}]
We may choose $s < n$ such that
\begin{equation}\label{eq:dimensionestimate}
    \mathrm{dim}_{\mathcal{H}}( E ) < n-s < \frac{ n }{ 1 + K \lambda(n) \binom{m}{n} }.
\end{equation}
Consider an essentially bounded $f \in W^{1,n}_{loc}( \Omega \setminus E, \mathbb{R}^m )$ satisfying
\begin{equation*}
    \| Df \|^n
    \leq
    K \star f^{\star}\omega
    \quad\text{almost everywhere}
\end{equation*}
for a calibration $\omega \in \Lambda^{n} \mathbb{R}^m$.

We wish to prove that $f \in W^{1,s}_{loc}( \Omega, \mathbb{R}^m )$ after which \Cref{thm:integrationbyparts} implies that $f \in W^{1,p}_{loc}( \Omega, \mathbb{R}^m )$ for every $p \in ( q( K, n, \|\omega\|_{\ell^1} ), p( K, n, \|\omega\|_{\ell^1} ) )$. Therefore $f$ has a continuous representative and as $E$ has negligible Lebesgue measure, the continuous representative is a $K$-quasiregular $\omega$-curve. So it suffices to establish $\psi f \in W^{1,s}( \Omega, \mathbb{R}^m )$ for an arbitrary nonnegative $\psi \in \mathcal{C}^{\infty}_{c}( \Omega )$.

Let $E'$ denote the intersection of $E$ and the support of $\psi$. Observe that the Sobolev $s$-capacity of $E'$ is zero by \eqref{eq:dimensionestimate} and e.g. \cite[Theorem 17.2.1]{Iw:Ma:01}. Therefore there exists a sequence of functions $\eta_i \in \mathcal{C}^{\infty}_c( \mathbb{R}^n )$ such that $\eta_i \equiv 1$ in an open neighbourhood of $E'$, $0 \leq \eta_i \leq 1$, $\lim_{ i \rightarrow \infty } \eta_i = 0$ almost everywhere on $\mathbb{R}^n$ and
\begin{align*}
    \lim_{ i \rightarrow \infty } \| D\eta_i \|_{ L^{s}( \Omega, \mathbb{R}^n ) } = 0. 
\end{align*}
We wish to apply \eqref{eq:cacciopolli:test} with $\phi_i = ( 1 - \eta_i ) \psi$ for each $i$; the sequence $\phi_i f$ is uniformly bounded in $L^{\infty}( \Omega, \mathbb{R}^m )$, compactly-supported in $E'$, and converge to $\psi f$ almost everywhere in $\Omega$. Moreover, \Cref{prop:caccioppoli:estimate} implies that
\begin{align}\label{eq:cacciopolli:test}
    \| \phi_i D f \|_{ L^{s}( \Omega, \mathbb{R}^{ m \times n } ) }^s
    &\leq
    C
    \| f \otimes D\phi_i \|_{ L^{s}( \Omega, \mathbb{R}^{ m \times n } ) }^s
    \leq
    C \| f \|_{ L^{\infty}(\Omega,\mathbb{R}^m) }^{s} \| D\phi_i \|_{ L^{s}(\Omega,\mathbb{R}^n) }^s
\end{align}
for a constant $C = C(s,n,K,\|\omega\|_1)$ for $i \in \mathbb{N}$.

We also have
\begin{align*}
    D( \phi_i f )
    =
    \phi_i Df
    +
    f \otimes D\phi_i
    =
    \phi_i Df
    +
    f \otimes ( (1-\eta_i) D\psi - \psi D\eta_i ) 
\end{align*}
where $f \otimes D\phi_i$ converge to $f \otimes D\psi$ in $L^{s}( \Omega, \mathbb{R}^{m \times n} )$ and $D\phi_i$ to $D\psi$ in $L^{s}( \Omega, \mathbb{R}^{ n } )$. So triangle inequality and the lower semicontinuity of energy imply that
\begin{align*}
    \| D( \psi f ) \|_{ L^{s}( \Omega, \mathbb{R}^{m \times n} ) }
    &\leq
    \liminf_{ i \rightarrow \infty }
    \| D( \phi_i f ) \|_{ L^{s}( \Omega, \mathbb{R}^{m \times n} ) }
    \\
    &\leq
    \liminf_{ i \rightarrow \infty }
    \| \phi_i D f \|_{ L^{s}( \Omega, \mathbb{R}^{ m \times n } ) }
    +
    \| f \otimes D\psi \|_{ L^{s}( \Omega, \mathbb{R}^{m \times n} ) }
    \\
    &\leq
    C^{1/s}
    \liminf_{ i \rightarrow \infty }
    \| f \|_{ L^{\infty}(\Omega, \mathbb{R}^m ) } \| D\phi_i \|_{ L^{s}(\Omega, \mathbb{R}^n) }
    +
    \| f \otimes D\psi \|_{ L^{s}( \Omega, \mathbb{R}^{m \times n} ) }
    \\
    &=
    C^{1/s}
    \| f \|_{ L^{\infty}(\Omega, \mathbb{R}^m) } \| D\psi \|_{ L^{s}(\Omega,\mathbb{R}^n) }
    +
    \| f \otimes D\psi \|_{ L^{s}( \Omega, \mathbb{R}^{m \times n} ) }
    <
    \infty.
\end{align*}
The claim $\psi f \in W^{1,s}( \Omega, \mathbb{R}^m )$ follows.
\end{proof}

\begin{proof}[Proof of \Cref{thm:pushforward}]
Let $\Omega \subset \mathbb{R}^n$ be an open set and for each $x_0 \in \Omega$, let $Q = Q( x_0, r_0 )$ denote the closed cube centered at $x_0$ and of side length $r_0$. Consider $r_0 > 0$ so that $2Q \subset \Omega$. Then \cite[Corollary 1.4 and Lemma 5.1]{Iko:23} imply the following: for almost every $s \in (0, 2r_0)$, we have that
\begin{equation}\label{eq:integralaverage}
    \int_{ Q( x_0, s) } f^{\star}\omega
    \leq
    A_{n-1}
    \left(
        \int_{ \partial Q( x_0, s ) } \|Df\|^{n-1}
    \right)^{ \frac{n}{n-1} },
\end{equation}
where $A_{n-1} = 1/( \omega_{n}^{ \frac{1}{n-1} } n^{ \frac{n}{n-1} } )$ is the constant in Almgren's isoperimetric inequality for integral currents \cite[Theorem 10]{Alm:86} where $\omega_n$ is the volume of the unit ball in $\mathbb{R}^n$. The integral over $Q( x_0, s )$ uses Lebesgue measure while the integral over $\partial Q( x_0, s )$ uses the surface measure. By Fubini's theorem, the set
\begin{equation*}
    \left\{
        r_0 < s < 2r_0
        \colon
        \int_{ \partial Q( x_0, s ) } \|Df\|^{n-1}
        \leq
        \frac{ 1 }{ r_0 }
        \int_{ 2Q \setminus Q }
            \|Df\|^{n-1}
    \right\}
\end{equation*}
has positive measure, so there is $s \in ( r_0, 2r_0 )$ such that \eqref{eq:integralaverage} holds and
\begin{equation*}
    \int_{ \partial Q( x_0, s ) } \|Df\|^{n-1}
    \leq
    \frac{ 1 }{ r_0 }
    \int_{ 2Q } \|Df\|^{n-1}.
\end{equation*}
Then \eqref{eq:integralaverage} and the distortion inequality imply that
\begin{equation*}
    \int_{ Q } \|Df\|^{n}
    \leq
    \int_{ Q( x_0, s) } \|Df\|^{n}
    \leq
    K A_{n-1}
    \left(
    \frac{ 1 }{ r_0 }
    \int_{ 2Q } \|Df\|^{n-1}
    \right)^{ \frac{n}{n-1} }.
\end{equation*}
Rearranging the inequality gives that
\begin{equation*}
    \aint{Q} \|Df\|^n
    \leq
    K A_{n-1} 2^{ \frac{n^2}{n-1} }
    \left(
        \aint{ 2Q } \|Df\|^{n-1}
    \right)^{ \frac{n}{n-1} }.
\end{equation*}
Since $2Q \subset \Omega$ was arbitrary, the Generalized Gehring lemma holds for the function $g = \|Df\|^{n-1}$. In particular, $g^{ 1 + \epsilon } \in L^{1}_{ loc }( \Omega )$ for $\epsilon = \epsilon( K, n ) > 0$; see \cite[Theorem 4.2]{Boj:Iwa:83} for an explicit estimate on the exponent $\epsilon$. The claim follows. 
\end{proof}

\begin{proof}[Proof of \Cref{thm:main:converse:intro}]
We consider a closed set $E \subset M$ removable for $W^{1,n}$-functions and a $K$-quasiregular $\omega$-curve $f \colon M \setminus E \to N$ having finite energy on precompact open sets of $M$. Recall that we are assuming that $\omega \in \Omega^{n}( N )$ is strongly nonvanishing, closed, and bounded.

Consider a compact exhaustion $( K_j )_{ j = 1 }^{ \infty }$ of $M$ for which the interior of $K_{j+1}$ contains $K_j$, and let $\phi_j \colon M \to [0,1]$ be a smooth map equal to one in $K_j$ and zero outside the interior of $K_{j+1}$ for $j \geq 1$. Next, consider an isometric embedding $\iota \colon N \xhookrightarrow{} L^{\infty}(N)$ and consider a Sobolev extension of $\phi_j ( \iota \circ f )$ for $j \geq 1$; the existence of the extension is a consequence of \Cref{prop:removability:mappings}. The pointwise limit of the sequence of extensions defines a Sobolev extension of $\iota \circ f$ and thus a Sobolev extension $F$ of $f$. The Sobolev extension $F$ has a $K$-quasiregular $\omega$-curve as a representative as follows by combining \Cref{prop:sectional:to:contractibility} with \Cref{thm:isoperimetric:to:continuity}; this is where we use the the sectional curvature upper bound and the injectivity radius lower bound on $N$. The first part of the claim follows.

Lastly, in case a $K$-quasiregular $\omega$-curve $f \colon M \setminus E \rightarrow N$ has a continuous extension $G \colon M \to N$ and $E$ satisfies the Hausdorff dimension bound in \Cref{thm:remov:bounded:curve}, we claim that $G$ is a $K$-quasiregular $\omega$-curve. A localization to charts yields the claim as follows. Consider $K' > K$ such that
\begin{equation}\label{eq:dimensionbound:proof}
    \mathrm{dim}_{ \mathcal{H} }( E )
    <
    \frac{ n }{ 1 + K' \lambda(n) \binom{m}{n} }
\end{equation}
and let $x_0 \in M$. Consider $R_0 > 0$ so small that there is a $(1+\epsilon)$-bi-Lipschitz diffeomorphism
\begin{equation*}
    \psi \colon B( G( x_0 ), R_0 ) \rightarrow B( 0, R_0 )
\end{equation*}
and that
\begin{equation*}
    \frac{1}{1+\epsilon} \|\omega\|( G( x_0 ) ) \leq \| \omega \|(x) \leq ( 1+\epsilon )\|\omega\|( G(x_0) )
    \quad\text{for every $x \in B( G( x_0 ), R_0 )$.}
\end{equation*}
Similarly, there is a small $r_0 > 0$ so that there is a $(1+\epsilon)$-bi-Lipschitz diffeomorphism
\begin{equation*}
    \phi \colon \overline{B}( x_0, r_0 ) \rightarrow \overline{B}( 0, r_0 )
\end{equation*}
between closed balls and $G( \overline{B}( x_0, r_0 ) ) \subset B( G( x_0 ), R_0 )$. We denote $H = \psi \circ G \circ \phi^{-1}$. Notice that $E' = \phi( \overline{B}( x_0, r_0 ) \cap E ) \subset \mathbb{R}^n$ is a closed set that satisfies the dimension bound \eqref{eq:dimensionbound:proof}.

Here $H|_{ B( 0, r_0) \setminus E' }$ is a $(1+\epsilon)^{4n}K$-quasiregular $\eta$-curve with respect to the differential form $\eta = ( \psi^{-1} )^{\star}\omega$. We argue as in the proof of \cite[Lemma 5.2]{Pan:20} to deduce that for small enough $r_0$, $H|_{ B( 0, r_0 ) \setminus E }$ is a $K'$-quasiregular $\eta_0$-curve with respect to the constant-coefficient differential form $\eta_0$: $\eta$ evaluated at $0$. The argument goes through because $H$ is continuous. Now \eqref{eq:dimensionbound:proof} and \Cref{thm:remov:bounded:curve} imply that $H|_{ B( 0, r_0 ) }$ is a $K'$-quasiregular $\eta_0$-curve. Then a simple covering argument implies that $G \in W^{1,n}_{loc}( M, N )$. Then $G$ being a $K$-quasiregular $\omega$-curve follows from $E$ having negligible measure.
\end{proof}
Theorems \ref{cor:quasiregular} and \ref{cor:pseudoholomorphic} are immediate corollaries of \Cref{thm:main:converse:intro}.
\begin{remark}
The reader observes that we only used the geometric assumptions on $N$ to deduce that the extension $F$ in the proof of \Cref{thm:main:converse:intro} has a continuous representative. It follows, therefore, that the conclusion of the claim holds under the assumptions of \Cref{thm:isoperimetric:to:continuity}. That is, if $N$ is complete and supports Euclidean isoperimetric inequalities for dimensions $n$ and $n-1$ up to some mass. Theorems \ref{cor:quasiregular} and \ref{cor:pseudoholomorphic} and the subsequent corollaries in \Cref{sec:geo:top:consi} can be extended in a similar manner.
\end{remark}

\begin{proof}[Proof of \Cref{cor:kunnethideal}]
We suppose that $f$ has finite energy and wish to show that $f$ is constant. By \Cref{thm:main:converse:intro}, there is a $K$-quasiregular $\omega$-curve $F \colon \mathbb{S}^n \rightarrow N$ extending $f$; it suffices to show that $F^{\star}\omega$ is weakly exact. Indeed, by the distortion inequality and Stokes' theorem,
\begin{equation*}
    \int_{\mathbb{S}^n} ( \|\omega\| \circ F ) \|DF\|^n \,d\mathcal{H}^n \leq K \int_{ \mathbb{S}^n } F^\star \omega = 0, \quad\text{cf. \cite[Section 2]{Haj:Iwa:Mal:Onn:08}}.
\end{equation*}
Hence $( \|\omega\| \circ F ) \|DF\|^n = 0$ $\mathcal{H}^{n}$-almost everywhere. As $\| \omega \| \circ F \neq 0$ everywhere, we deduce that $\|DF\|^n = 0$ $\mathcal{H}^{n}$-almost everywhere. Thus $F$ is constant.

To finish the proof, we use the definition of the Künneth ideal, and observe that
\begin{equation*}
    F^{\star}\omega
    =
    d
    \left(
        F^{\star}\eta
    \right)
    +
    \sum_{ i = 1 }^{ l } F^{\star}\alpha_i \wedge F^{\star}\beta_i
    \quad\text{in the weak sense.}
\end{equation*}
Here $F^{\star}\alpha_i$ and $F^{\star}\beta_i$ are $(n/n_i)$- and $(n/(n-n_i))$-integrable, respectively, and weakly closed. We recall that the conformal cohomology of $\mathbb{S}^n$ is trivial between dimensions $1$ and $n-1$ (see \cite[Section 3.2]{Kan:Pan:19} for the definition and \cite[Theorem 4.1]{Kan:Pan:19} for the equivalence statement with the usual cohomology). Hence, if $n_i > 1$, there is an $(n_i-1)$-form $\sigma_i$ that is $(n/(n_i-1))$-integrable and whose weak differential is $F^\star \alpha_i$. For $n_i = 1$, the same conclusion holds but $\sigma_i \in L^{p}( \mathbb{S}^n )$ for every $1 \leq p < \infty$. Either way, $F^{\star}\alpha_i \wedge F^{\star}\beta_i = d( \sigma_i \wedge F^{\star}\beta_i )$ in the weak sense, so $F^{\star}\omega = d( F^{\star}\eta + \sum_{ i = 1 }^{ l } \sigma_i \wedge F^{\star}\beta_i )$ in the weak sense and thus $F^{\star}\omega$ is weakly exact.
\end{proof}
\begin{proof}[Proof of \Cref{cor:kunnethideal:topological}]
The conclusion is simply a reformulation of \Cref{cor:kunnethideal}.
\end{proof}

\section*{Declarations}
\subsection*{Conflict of interest} The author declares that they have no conflict of interest.

\subsection*{Financial support}
This work was supported by the Academy of Finland, project number 332671. Views and opinions expressed are those of the author only and do not necessarily reflect those of the Academy of Finland or the other funding organizations.

\subsection*{Data}
Data sharing not applicable to this article as no datasets were generated or analysed during the current study.

\newcommand{\etalchar}[1]{$^{#1}$}

\end{document}